\documentclass[10pt]{article}
\font\smallit=cmti10

\usepackage{amssymb,latexsym,amsmath,epsfig,amsthm,comment} %% Add other packages as necessary

\makeatletter

\renewcommand\section{\@startsection {section}{1}{\z@}
{-30pt \@plus -1ex \@minus -.2ex}
{2.3ex \@plus.2ex}
{\normalfont\normalsize\bfseries\boldmath}}

\renewcommand\subsection{\@startsection{subsection}{2}{\z@}
{-3.25ex\@plus -1ex \@minus -.2ex}
{1.5ex \@plus .2ex}
{\normalfont\normalsize\bfseries\boldmath}}

\renewcommand{\@seccntformat}[1]{\csname the#1\endcsname. }

\makeatother

\newtheorem{theorem}{Theorem}
\newtheorem{lemma}{Lemma}
\newtheorem{proposition}{Proposition}

\theoremstyle{definition}

\renewcommand{\mod}[1]{\ (\text{mod}\ #1)}

\begin{document}

\begin{center}
\uppercase{\bf On Upper Bounds for the Count of Elite Primes}
\vskip 20pt
{\bf Matthew Just}\\
{\smallit Department of Mathematics, University of Georgia, Athens, Georgia 30605, United States}\\
{\tt justmatt@uga.edu}\\ 
%\vskip 10pt
%{\bf Author Two\footnote{any footnote here}}\\
%{\smallit Department of Mathematics, University of Two, Anywhere AAAAA, Country}\\
%{\tt you@math.two.edu}\\ (optional)
\end{center}

\centerline{\bf Abstract}
\noindent
We look at upper bounds for the count of certain primes related to the Fermat numbers $F_n=2^{2^n}+1$ called elite primes. We first note an oversight in a result of K{\v{r}}{\'\i}{\v{z}}ek, Luca and Somer and give the corrected, slightly weaker upper bound. We then assume the Generalized Riemann Hypothesis for Dirichlet $L$ functions and obtain a stronger conditional upper bound.

\section{Introduction}

The Fermat numbers are given by $F_n=2^{2^n}+1$ for $n\geq 0$. The first few Fermat numbers are 3, 5, 17, 257, 65537, 4294967297. Notice that the first five Fermat numbers are prime, and it was initially conjectured (by Fermat) that all such numbers are prime. The sixth Fermat number is not prime, and no other Fermat primes are known. It is known that $F_n$ is composite for $5\leq n \leq 32$ though interestingly no prime factor of $F_{14}$, $F_{20}$, $F_{22}$, or $F_{24}$ is  known (see \cite{crandall2003twenty}).
An efficient test exists to determine whether or not a Fermat number is prime, called P\'epin's test. 
\begin{proposition}[P\'epin's test \cite{ribenboim2012new}]
Let $n>0$. Then $F_n$ is prime if and only if \[3^{(F_n-1)/2} \equiv -1 \emph{$\mod{F_n}$}.\] 
\end{proposition}

\begin{comment}
\begin{proof}
First suppose that $F_n$ is prime. Then by Euler's criterion $3^{(F_n-1)/2} \equiv -1 \mod{F_n}$ if and only if 3 is a quadratic nonresidue$\mod{F_n}$. That is to say \[\left(\frac{3}{F_n}\right) = -1\] where here we use the Legendre symbol defined for odd primes $p$ by \[\left(\frac{a}{p}\right)=\begin{cases}0 & \text{if $p\mid a$} \\ 1 & \text{if $p\nmid a$ and $a$ is a quadratic residue$\mod{p}$ } \\-1 & \text{if $p\nmid a$ and $a$ is a quadratic nonresidue$\mod{p}$ }\end{cases}\] Now for $n>0$ it is clear that $F_n \equiv 1\mod{4}$ so by the law of quadratic reciprocity \[\left(\frac{3}{F_n}\right) = \left(\frac{F_n}{3}\right)\] so it is enough to show that $F_n$ is a quadratic nonresidue$\mod{3}$. But for $n>0$ it is again clear $F_n\equiv -1 \mod{3} $, so $F_n$ is a quadratic nonresidue$\mod{3}$.

Now if $3^{(F_n-1)/2}\equiv -1 \mod{F_n}$ let $f$ be the order of $3\mod{F_n}$. Then $f\mid (F_n -1) = 2^{2^n}$. But it is also that case that $f\nmid (F_n-1)/2 = 2^{2^n-1}$, so $f=2^{2^n}$. Thus the order of the multiplicative group of units$\mod{F_n}$ has order $2^{2^n}=F_n-1$ which can only happen if $F_n$ is prime. \end{proof}
\end{comment}

The only fact about the prime 3 used in this proof is that $F_n$ is a quadratic nonresidue$\mod{3}$ for $n>0$. Thus we can replace 3 with any other prime that satisfies this requirement. 
%We can also consider primes $a$ that satisfy the slightly weaker requirement that there exists an $n_0$ such that for all $n>n_0$ the $n$th Fermat number is a quadratic nonresidue$\mod{a}$.
Primes $p$ for which all sufficiently large Fermat numbers are quadratic nonresidues$\mod{p}$ are called \textit{elite primes}, a term introduced by Aigner \cite{aigner1986primzahlen}. Thus we can use elite primes to test the primality of $F_n$ for all but finitely many Fermat numbers. For more on the search for elite primes see \cite{allep}, \cite{continuingsearch}, and \cite{complex}. 

Similarly defined are anti-elite primes which are primes $p$ such that $F_n$ is a quadratic residue for all but finitely many $n$. Generalizations of elite and anti-elite primes for numbers other than the Fermat numbers were studied by M{\"u}ller, see \cite{muller2007anti}, \cite{muller2008generalization}, and \cite{muller2008generalized}.

Let $E(x)$ be the number of elite primes up to $x$. K{\v{r}}{\'\i}{\v{z}}ek, Luca and Somer \cite{kvrivzek2002convergence} gave the upper bound \[E(x) = O\left( \frac{x}{(\log x)^2} \right)\] as $x\rightarrow\infty$.  Their proof used the estimate  \[\prod_{i=0}^{2t}F_i< 2^{2^{t+1}};\] however one can check that for all $t\geq 0$ \[\prod_{i=0}^{2t}F_i = 2^{2^{2t+1}} -1.\] Using the revised estimate $\prod_{i=0}^{2t}F_{i} < 2^{2^{2t+1}}$ we obtain the following result following their proof: 
\begin{theorem}
    Let $E(x)$ be the number of elite primes up to $x$. Then \[E(x)=O\left(\frac{x}{(\log x)^{3/2}}\right)\] as $x\rightarrow \infty$.
\end{theorem}

Corollary 2 of K{\v{r}}{\'\i}{\v{z}}ek, Luca and Somer's paper states that the sum of the reciprocals of the elite primes converges. We note that this still follows from Theorem 1.

We next give our main result which is a stronger upper bound for $E(x)$ conditional upon the Generalized Riemann Hypothesis for Dirichlet $L$ functions:

\begin{theorem}
    Let $E(x)$ be the number of elite primes up to $x$. Then assuming the Generalized Riemann Hypothesis \[E(x) = O(x^{5/6})\] as $x\rightarrow \infty$.
\end{theorem}

\section{Proof of Theorem 2}

Before beginning the proof we introduce some tools we will use. Let $\left(\frac{a}{m}\right)$ denote the Kronecker symbol.

\begin{comment}
The Jacobi symbol, a generalization of the Legendre symbol, is defined for odd numbers $m=p_1^{\alpha_1}p_2^{\alpha_2}\ldots p_k^{\alpha_k}$ by \[\left( \frac{a}{m} \right) = \prod_{i=1}^k \left(\frac{a}{p_i} \right)^{\alpha_i}\] We will use the following version of the law of quadratic reciprocity that states if $a$ and $m$ are odd coprime integers and either $a$ or $m$ is congruent to $1\mod{4}$ then \[\left( \frac{a}{m} \right) = \left( \frac{m}{a} \right)\] 

The Kronecker symbol, a generalization of the Jacobi symbol, is defined for all integers $a$ and $m$ (which we write as $m=u2^s k$ where $u=\pm 1$ and $k$ is odd) by \[\left( \frac{a}{m} \right) = \left( \frac{a}{u}\right) \left( \frac{a}{2}\right)^s \left( \frac{a}{k} \right)\] where \[\left(\frac{a}{u}\right) = \begin{cases}1 & \text{if $u=1$ or $a<0$} \\ -1 & \text{if $u=-1$ and $a\geq 0$} \end{cases}\] and \[\left(\frac{a}{2}\right)=\begin{cases}0 & \text{if $2\mid a$} \\ 1 & \text{if $a\equiv \pm1\mod{8}$} \\ -1 & \text{if $a\equiv\pm3\mod{8}$}\end{cases}\] Finally if $m=0$ then \[\left(\frac{a}{0}\right) = \begin{cases}1 & \text{if $a=\pm 1$} \\ 0 & \text{otherwise} \end{cases}\] 
\end{comment}

We will use the fact that for a fixed integer $a$ satisfying $a\neq 0$ and $a\not\equiv 3\mod{4}$ the Kronecker symbol is a Dirichlet character to the modulus $|a|$ (or $4|a|$ if $a\equiv2\mod{4}$). Furthermore this Dirichlet character is not the principal character as long as $a$ is not a square. The key lemma will be the following classical result which can be found in Montgomery and Vaughan \cite{montgomery2007multiplicative}.

\begin{lemma}
    Let $\chi$ be a Dirichlet character to the modulus $q$, not the principal character. Then assuming the Generalized Riemann Hypothesis \[\sum_{p\leq x} \chi(p) = O(\sqrt{x} \log(qx)).\]
\end{lemma}

\begin{proof}[Proof of Theorem 2]
    Let $p\leq x$ be an elite prime. We write \[p-1 = 2^{e_p} k_p, \] where $k_p$ is odd and let $f_p$ denote the multiplicative order $2\mod{k_p}$. Now for any $m\geq e_p$ we have that \[(p-1)\mid 2^m (2^{f_p}-1)=2^{m+f_p}-2^m,\] which shows by Fermat's little theorem that \[2^{2^{m+f_p}}\equiv 2^{2^m}\mod{p}.\] This periodicity of the sequence of Fermat numbers shows that if there exists an $m\geq e_p$ such that $F_m$ is a quadratic residue$\mod{p}$ then $p$ cannot be elite since $F_{m+\ell f_p}$ would be a quadratic residue$\mod{p}$ for all $\ell \geq 0$. 
    
    Now suppose $p$ is a prime with $e_p> t$, where $t$ is a parameter depending on $x$ to be chosen later. Then $p$ lies in the residue class $1\mod{2^t}$. As long as $2^t\leq \sqrt{x}$ then we may apply the Brun-Titchmarsh inequality to get an upper bound on the distribution of such primes in arithmetic progressions: \[\pi(x;2^t,1) \leq \frac{2x}{\phi(2^t) (\log x - \log 2^t)} \ll \frac{x}{2^t}. \]  
    
    Now we assume $p$ is an elite prime with $e_p\leq t$. We see now that $p$ must be a quadratic nonresidue modulo $F_{t+i}$ for all $i\geq 0$. Looking at the Legendre symbol we see that for any prime number $p$ \[ 1 - \left( \frac{F_{t+i}}{p}  \right) = \begin{cases}2 & \text{if $F_{t+i}$ is a quadratic nonresidue$\mod{p}$}, \\ 0 & \text{if $F_{t+i}$ is a quadratic residue$\mod{p}$} , \\ 1 & \text{if $p\mid F_{t+i}$}.\end{cases} \] 
    Fixing another parameter $T$ depending on $x$ to be chosen later let \[A=\prod_{i=0}^T F_{t+i}\] so that we now have the following upper bound for $E(x)$: \[E(x) \leq  \frac{1}{2^{T+1}}\sum_{p\leq x }\prod_{i=0}^T \left(1-\left(\frac{F_{t+i}}{p}\right) \right) + \sum_{p|A} 1 + O\left(\frac{x}{2^t}\right).\] Notice that \[A<\prod_{i=0}^{t+T} F_i <2^{2^{T+t+1}},\] and thus \[\sum_{p|A}1 \ll \log A \ll 2^T2^t. \] 
    As for the first term in our estimate for $E(x)$, we have 
    \[\frac{1}{2^{T+1}}  \sum_{j=1}^{T+1} (-1)^j \sum_{\substack{B\subset \{0,1,2,\ldots,T \}\\ |B|=j}} \sum_{p\leq x}\left(\frac{\prod_{b\in B} F_{t+b}}{p} \right) + \frac{\pi(x)}{2^{T+1}}.\] Using the fact that these inner Kronecker symbols are Dirichlet characters to the modulus \[\prod_{b\in B} F_{t+B} \leq A < 2^{2^{t+T+1}}, \] we can apply Lemma 1 once we observe that Fermat numbers are pairwise coprime and hence any product of Fermat numbers will never be a square. We then have the upper bound \[\sum_{p\leq x}\left(\frac{\prod_{b\in B} F_{t+b}}{p} \right) \ll \sqrt{x}(\log x + 2^{t+T}). \] Putting all this together, we have \[E(x) \ll \sqrt{x}\log x + \sqrt{x}2^{t+T} + \frac{x}{2^T} + \frac{x}{2^t}.\]
    Letting $t=T=\frac{\log x}{6\log 2}$ gives the desired result. \end{proof}

    \noindent{\bf Remark.}    Theorem 2 gives an upper bound for the count of elite primes, but if we instead used \[1 + \left( \frac{F_{t+i}}{p}  \right) = \begin{cases}2 & \text{if $F_{t+i}$ is a quadratic residue$\mod{p}$} \\ 0 & \text{if $F_{t+i}$ is a quadratic nonresidue$\mod{p}$}\\ 1 & \text{if $p\mid F_{t+i}$},\end{cases}\] we obtain the same upper bound for the count of anti-elite primes. Furthermore, if we consider the generalized Fermat numbers to the base $b$, \[b^{2^n} + 1,\] we can ask what is special about the case $b=2$. In fact the only time the base $b=2$ is used in the proof of Theorem 2 is that for a fixed prime $p$ the sequence $2^{2^n}+1$ will eventually be periodic$\mod{p}$. But this is true for any base $b$, hence Theorem 2 will also hold for the count of generalized elite and anti-elite primes with respect to the generalized Fermat numbers to the base $b$.
    
    \section*{Acknowledgments}

The author was partially supported by the Research and Training Group grant DMS-1344994 funded by the National Science Foundation. He thanks Paul Pollack and Florian Luca for helpful comments.

\end{document}